\documentclass[10 pt, reqno]{amsart}

\usepackage[utf8]{inputenc}
\usepackage{amsmath, amssymb, amsthm}
\usepackage{mathrsfs}
\usepackage{booktabs} 
\usepackage[margin=1in]{geometry} 
\usepackage{hyperref} 
\usepackage{xurl} 
\usepackage{enumitem} 

\newcommand{\M}{\mathcal{M}}

\newcommand{\lk}{\lambda_k}
\newtheorem{theorem}{Theorem}[section]
\newtheorem{lemma}[theorem]{Lemma}
\newtheorem{proposition}[theorem]{Proposition}
\newtheorem{corollary}[theorem]{Corollary}
\theoremstyle{definition}
\newtheorem{definition}{Definition}

\newtheorem{remark}[theorem]{Remark}
\numberwithin{equation}{section}

\newcommand{\MMM}{\mathcal{M}}

\newcommand{\lbtheta}{(1-c_\MMM r)}

\newcommand{\norm}[1]{\left\lVert#1\right\rVert}

\usepackage{tikz}
\begin{document}

\title{Dirichlet Eigenvalue Approximation on Manifolds with Cylindrical Boundary }
\author{Anusha Bhattacharya}
\address{Department of Mathematical Sciences, Indian Institute of Science Education and Research Mohali, \newline Sector 81, SAS Nagar, Punjab- 140306, India.}
\email{ph21059@iisermohali.ac.in}
\begin{abstract}
We prove that the Dirichlet eigenvalues of the Laplace–Beltrami operator on a compact Riemannian manifold with cylindrical boundary can be approximated by the spectrum of truncated graph Laplacians constructed from $(\varepsilon,\rho) $-proximity graphs on the manifold. The approximation is uniform over a class $\mathcal{M}$ of manifolds,  characterized by bounds on Ricci curvature, a lower bound on the injectivity radius, and an upper bound on the diameter. We show that the $k$-th  eigenvalue of the truncated graph Laplacian lies between  $k$-th Dirichlet eigenvalues of truncated domains of the manifold. As the parameters $\varepsilon, \rho$ and the ratio $\frac{\varepsilon}{\rho}$ tend to zero, these estimates yield convergence of the eigenvalues of the truncated graph Laplacian  to the Dirichlet eigenvalues of the Laplace–Beltrami operator. 

\end{abstract}
\maketitle
\section{Introduction}

A fundamental problem in spectral geometry is to approximate the spectrum of the Laplace-Beltrami operator on a Riemannian manifold by finite discrete models. Given a finite sampling of points from the manifold, a natural approach is to construct a weighted graph and an associated graph Laplacian on the discrete data. Graph Laplacians are extensively used in unsupervised learning tasks, like dimensionality reduction \cite{belkin2003, coifman},  semi-supervised learning \cite{Belkin2004, calder2020} and spectral clustering \cite{andrew, vonLuxburg2008}. The convergence of learning algorithms in these settings depends on spectral convergence of the graph Laplacian to the Laplace-Beltrami operator.

\par Existing results on spectral convergence can broadly be divided into probabilistic and deterministic settings. For points sampled according to a known probability measure on the manifold, spectral convergence of graph Laplacian in probability has been studied in \cite{NIPS2006_5848ad95, trillos2020}. For manifolds with boundary, Peoples and Harlim \cite{peoples} studied the Dirichlet eigenvalue problem and proved spectral convergence of graph Laplacians with high probability as the number of uniformly sampled points increases. 

\par Several works \cite{aubry2013approximation, fujiwara, tatiana},  have established spectral convergence in the deterministic setting. Burago, Ivanov, and Kurylev  \cite{burago} proved spectral convergence of graph Laplacians constructed from $\varepsilon$-nets on compact Riemannian manifolds without boundary under uniform bounds on sectional curvature, injectivity radius, and diameter. Their methods were later extended in joint work with the author \cite{MaityBhattacharya2026} to a broader class of Riemannian manifolds by replacing sectional curvature bounds with bounds on Ricci curvature. Extending these spectral approximation results to manifolds with boundary introduces additional challenges, as neighborhoods of the boundary, although small in volume, can significantly influence the associated energy estimates \cite{pmlr-v23-belkin12}. Lu \cite{lu} has derived spectral approximations of the Laplace-Beltrami operator with the Neumann boundary condition using reflected geodesics and the techniques of  \cite{burago}.

\par In this paper, we approximate the Dirichlet spectrum of Laplace-Beltrami operator on compact Riemannian manifolds with cylindrical boundaries using truncated graph Laplacians on suitable discretizations of the manifold. The cylindrical boundary induces a product structure in a neighborhood of the boundary, which allows us to construct  discretizations that  control  energy estimates near the boundary. We can then apply the variational techniques for spectral approximation developed in \cite{burago}. 

\begin{definition}
Let $(M,g)$ be a compact Riemannian manifold with boundary $\partial M$.  $M$ has a \emph{cylindrical boundary of width $s_0>0$} if there exists a neighborhood $U$ of $\partial M$ in $M$ such that $(U,g)$ is isometric to the product manifold $(\partial M \times [0,s_0],\, g_{\partial M} + dt^2)$, where $g_{\partial M}$ denotes the induced metric on $\partial M$ and $t$ denotes the coordinate on $[0,s_0]$.
\end{definition}

Consider the class of $\M$ of  $ n$-dimensional smooth Riemannian manifolds  with cylindrical boundary of width greater than $s_0$ satisfying:
 $$|Ric|\leq \lambda, \quad diam \leq D, \quad {\rm and} \quad inj\geq i_0$$
 where $inj$ and $diam$ denote the injectivity radius and diameter of the manifold, respectively. For $M\in \M$ and $t\le s_0$, we consider truncated domains 
\[M_t=\{x\in M| d(x,\partial M)\ge t\}\]
where $\partial M$ denotes the boundary of $M$.\\

\subsection{Discretization Scheme} \label{discretization_scheme}

Let $\partial X_0 \subset \partial M$ be a finite $\varepsilon$-net of the boundary, and let $\nu$ denote the inward unit normal vector field along $\partial M$. We project $\partial X_0$ on  $\partial M_t$ along the normal direction $\nu$. Precisely, for $t>0$ we define
\[
\partial X_t = \{ \exp_x(t\nu_x) : x \in X_0 \}.
\]
We observe that $\partial X_t \subset \partial M_t$. Fix $0<t+\varepsilon<s_0$. We construct an $\varepsilon$-net $X$ in $M$ as follows. 
Let $X_{{t+\varepsilon}}$ be an $\varepsilon$-net on $M_{t+\varepsilon} $ constructed by extending  $\partial X_{t+\varepsilon}$ to an the interior of $M_{t+\varepsilon}$. 
For $k\in\mathbb{N}\cup \{0\}$ with $t-k\varepsilon\ge0$, we define

\[
X = X_{t+\varepsilon} \bigsqcup \partial X_{t+\varepsilon} 
\bigsqcup_{k:\, t-k\varepsilon>0} \partial X_{t-k\varepsilon}.
\]

 The following figure illustrates the discretization of the cylindrical region $\partial M \times [0,t]$.
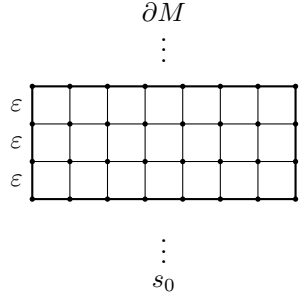
\begin{center}

\begin{tikzpicture}[scale=0.5]

\def\eps{1}
\def\nx{7}   
\def\ny{3}   
\draw[thick] (0,0) rectangle (\nx*\eps,\ny*\eps);

\foreach \i in {1,2,3,4,5,6} {
    \draw (\i*\eps,0) -- (\i*\eps,\ny*\eps);
}

\foreach \j in {1,2} {
    \draw (0,\j*\eps) -- (\nx*\eps,\j*\eps);
}

\foreach \i in {0,1,2,3,4,5,6,7} {
    \foreach \j in {0,1,2,3} {
        \fill (\i*\eps,\j*\eps) circle (2pt);
    }
}

\node[left] at (0,0.5*\eps) {$\varepsilon$};
\node[left] at (0,1.5*\eps) {$\varepsilon$};
\node[left] at (0,2.5*\eps) {$\varepsilon$};

\node[above] at (0.5*\nx*\eps,\ny*\eps+0.4) {$\vdots$};
\node[above] at (0.5*\nx*\eps,\ny*\eps+1.5) {$\partial M$};
\node[below] at (0.5*\nx*\eps,-0.4) {$\vdots$};
\node[below] at (0.5*\nx*\eps,-1.8) {$s_0$};

\end{tikzpicture}
\end{center}
The layer $\partial X_{t+\varepsilon}$ is included so that the discretization of functions introduced in the subsequent sections, which are compactly supported in $M_{t+\frac{\varepsilon}{2}}$, vanishes on $\partial X_t$. Consequently, the resulting discrete functions satisfy the Dirichlet boundary condition on the truncated graph.

\begin{definition}[\bf{$(\varepsilon,\rho)$-proximity graph}] Given any $0<\varepsilon<3\rho<t+\varepsilon<s_0$, a finite graph $\Gamma^t_{\varepsilon,\rho}$ with the vertex set $ X=\{x_i\}_{i=1}^N$ from \ref{discretization_scheme} and edges ${e_{ij}}$ is called an $(\varepsilon,\rho)$-proximity graph of $(M,g)$ if the following conditions hold. 
\begin{itemize}
        \item $M$ can be partitioned into measurable subsets using a Voronoi decomposition $\{V_i\}_{i=1}^N$ on $X$ such that $V_i\subset B_{\varepsilon}(x_i)$ and $M=\bigcup B_{\varepsilon}(x_i)$ where $B_{\varepsilon}(x_{i})$ is the ball in $M$ centered at $x_{i}$ with radius $\varepsilon$. 
        \item Two vertices $x_i$ and $x_j$ are connected by an edge $e_{ij}$ if the Riemannian distance between them is less than $\rho.$
        \item {\it Measure on $\Gamma^t_{\varepsilon,\rho}$}:  Let $\mu_{i}=\text{vol}(V_{i}).$  $X$ can then be equipped with a discrete measure $\mu=\sum_{i=1}^{N}\mu_{i}\delta_{x_{i}}$ where $\delta_{x_i}$ denotes the Dirac measure at $x_i.$
         \item \textit{Weights on edges}: Let $\nu_{n}$ be the volume of the unit ball in the Euclidean $n$-space. To an edge $e_{ij}$, assign the weight 
    \begin{equation*}
        w_{ij}=\frac{2(n+2)}{\nu_{n}\rho^{n+2}}\mu_{i}\mu_{j}.
    \end{equation*}
        \end{itemize}
        \end{definition}
\begin{remark}
Since the boundary neighborhood is isometric to a product cylinder, the Voronoi cells in the cylindrical region have a product structure.

Precisely, let $W_i^0$ denote the Voronoi cell on  $\partial M$ corresponding to $x_i \in \partial X_0$. Then taking the projection of $W_i^0$ along $\nu$ on $\partial M_{t'}$, we define $W_i^{t'}=\exp(t'\nu(W_i^0))$, for all $k\in \mathbb{N}\cup\{0,-1\}$ such that $t'= t-k\varepsilon>0$. 
Then the associated Voronoi cell of $\exp_{x_i}(t'\nu_{x_i})$in $M$ is given by
\[
V_i^{t'} = W_i^{t'} \times \left[-\frac{\varepsilon}{2},\frac{\varepsilon}{2}\right].
\]

The structure of the Voronoi cells near the boundary is essential for controlling the energy estimates. 
\end{remark}
The inner product on the  space $L^2(X_t)=L^2(X_t,\mu)$ is given by
\[\langle u,v\rangle= \langle u,v\rangle_{L^2(X_t)}=\sum_{x_i\in X_t}\mu_i u(x_i)v(x_i).\]

\begin{definition}
 The truncated graph Laplacian $\Delta_t: L^2(X_t)\rightarrow L^2(X_t)$ is defined as
\[
(\Delta_t u)(x_i)
=
\frac{1}{\mu_i}
\sum_{x_j\in X: x_i\sim x_j} w_{ij}\big(u(x_j)-u(x_i)\big).
\]
\end{definition}
We observe that $-\Delta_t$ is a self-adjoint operator on $L^2(X_t)$ and is non-negative definite.  \\
\textbf{Convention.}
We identify functions $u : X_t \to \mathbb{R}$ with their extension by zero to $X$. In particular, $u(x_j)=0$ whenever $x_j \notin X_t$. Similarly, we identify $L^2(M_t)$ with the subspace of $L^2(M)$ consisting of functions extended by zero outside $M_t$. Throughout the paper, $c_n$ and $C_{\mathcal{M}}$ denote constants depending only on $n$ and the class $\mathcal{M}$.

\par The following theorem is the main result of the paper. It provides two-sided comparison estimates between eigenvalues of the truncated graph Laplacian and Dirichlet eigenvalues of suitable truncated
domains of the manifold.

\begin{theorem}
\label{thm:main}
For $0 < \varepsilon < \rho < t < s_0$,  consider a manifold $M\in \M$ . Let $\lk(M_{t+\varepsilon})$ and $\lk(M_{t+2\varepsilon-\rho})$ denote the Dirichlet eigenvalues of the Laplace-Beltrami operator restricted to the truncated domains $M_{t+\varepsilon}$ and $M_{t+2\varepsilon-\rho}$ of $M$ respectively, and $\lk(\Gamma^t_{\varepsilon,\rho})$ and $\lk(M_t)$ be the eigenvalues of $-\Delta_t$ and $-\Delta_M$ with the Dirichlet boundary condition. Then, there exist a constant $C_\M>0$ such that for any $\rho,\frac{\varepsilon}{\rho}<\frac{1}{C_\M}$,

\[
\left(1-C_{\M,k}\left(\frac{\varepsilon}{\rho}+\rho\right)\right)\lk(M_{t+2\varepsilon-\rho})\le\lk(\Gamma^t_{\varepsilon,\rho})\le \left(1+C_{\M,k}\left(\frac{\varepsilon}{\rho}+\rho\right)\right)\lk(M_{t+\varepsilon}).
\]
\end{theorem}
The comparison estimates in the theorem show that eigenvalues of the truncated graph Laplacian are controlled by the Dirichlet eigenvalues of two nearby truncated domains of the manifold. In the proximity graph $\Gamma_{\varepsilon,\rho}^t$, the parameter $\varepsilon$ controls density of point sampling, while $\rho$ determines the scale of  local interactions in the graph. As these parameters decrease, the discretization becomes finer and the discrete Dirichlet energy of functions (\ref{eq:discrete_dirichlet_energy}) supported on $X_t$ increasingly reflects the Dirichlet energy on $M_t$. Consequently, eigenvalues of the truncated graph Laplacian converge to Dirichlet eigenvalues of the truncated manifold.
\begin{corollary}
\label{corollary}
For each $k\in\mathbb{N}$,
\[
\lambda_k(\Gamma_{\varepsilon,\rho}^t)\to\lambda_k(M_t)
\quad\text{as }\varepsilon,\rho\to0,\ \frac{\varepsilon}{\rho}\to0.
\]
\end{corollary}

The truncated domains $M_t$ are obtained by removing a boundary layer of width $t$ from the manifold. As $t\to 0$, these domains exhaust $M$. Hence, the Dirichlet eigenvalues on $M_t$  converge to the  Dirichlet eigenvalues on $M.$
\begin{proposition}
\label{proposition}
 For each $k\in\mathbb{N}$,
\[
\lambda_k(M_t)\to\lambda_k(M)
\quad\text{as }t\to0.
\]   
\end{proposition}
\subsection*{Idea of proof}
To prove Theorem \ref{thm:main}, we consider truncations of the manifold within the cylindrical radius and apply the spectral approximation techniques developed in \cite{burago}. The choice of point sampling near the boundary ensures that the Voronoi cells can be effectively described in the presence of the cylindrical boundary. Moreover, the truncations can be chosen so that the discretization and interpolation maps avoid any interaction with the boundary.

The proof  is based on the Min–Max principle.We construct a discretization map on the truncated manifold that allows us to relate the discrete Dirichlet energy on the graph to the energy functional on the manifold. By controlling the Rayleigh quotient on the image under this discretization map of the span of the first $k$ eigenfunctions of $-\Delta$, we obtain an upper bound for the $k$-th eigenvalue of the truncated graph Laplacian.

The lower bound for the $k$-th eigenvalue of the truncated graph Laplacian is obtained using an interpolation map that lifts functions from $L^2(X_t)$ to $C^{0,1}(M)$. This map acts as an approximate inverse of the discretization map. We establish estimates relating the derivative of the interpolated function to the discrete Dirichlet energy. Restricting the Rayleigh quotient to the image of the interpolation map applied to the span of the first $k$ eigenfunctions of the truncated graph Laplacian, and using these interpolation estimates, yields the lower bound.
We then use domain monotonicity and continuity of Dirichlet eigenvalues under monotone sequence of truncations to obtain spectral convergence of the graph Dirichlet eigenvalues to the Dirichlet eigenvalues of the manifold.

\section{Preliminaries}

In this section, we introduce the discrete Dirichlet energy associated with the graph discretization and show that it corresponds to the quadratic form of the truncated graph Laplacian $\Delta_t$. This connection allows the eigenvalues of $\Delta_t$ to be studied using variational characterizations analogous to those for the Dirichlet Laplace-Beltrami operator.

For an $(\varepsilon,\rho)$-proximity graph $\Gamma^t_{\varepsilon,\rho}$, let $E=E(\Gamma^t_{\varepsilon,\rho})$ denote the set of directed edges of the graph. On $L^2(E)$, we have the following inner product:
\[\langle \xi,\eta\rangle=\frac{1}{2}\sum_{e\in E}w(e)\xi(e)\eta(e).\]
\begin{definition}
    Let $u\in L^2(X)$. The discrete differential $\delta u:E\rightarrow\mathbb{R}$ is defined as 
    \[
\delta u(e_{ij}) := u(x_j)-u(x_i).
\]
\end{definition}

The corresponding Dirichlet energy is
\begin{equation}
    \label{eq:discrete_dirichlet_energy}
    \|\delta u\|^2
=
\frac{n+2}{\nu_n \rho^{n+2}}
\sum_{x_i \sim x_j}
\mu_i \mu_j \, |u(x_i)-u(x_j)|^2 .
\end{equation}

To approximate the Dirichlet eigenvalues on $M_t$, we consider functions on the graph that are supported in $X_t$ and extend them by zero outside $X_t$. Then the edges whose endpoints both lie outside $X_t$ do not contribute to the discrete Dirichlet energy. Thus the energy depends only on edges intersecting $X_t$, which naturally incorporates the Dirichlet boundary condition. 
\begin{lemma}
\label{lem:discrete-green}
For all $u,v \in L^2(X_t)$,
\[
\langle \delta u, \delta v \rangle_{L^2(E)}
=
-\langle \Delta_\Gamma u, v \rangle_{L^2(X_t)} .
\]
\end{lemma}

\begin{proof}

\begin{align*}
\langle \delta u, \delta u \rangle_{L^2(E)}
&=
\frac{1}{2}\sum_{x_i\sim x_j}
w_{ij}\bigl(u(x_j)-u(x_i)\bigr)^2 \\
&=
\frac{1}{2}\sum_{x_i\sim x_j}
w_{ij}\bigl(u(x_i)-u(x_j)\bigr)u(x_i)
-
\frac{1}{2}\sum_{x_i\sim x_j}
w_{ij}\bigl(u(x_i)-u(x_j)\bigr)u(x_j).
\end{align*}

Since the edges are directed, interchanging the indices $i$ and $j$ in the second sum and using the symmetry
$w_{ij}=w_{ji}$, we obtain
\[
\langle \delta u, \delta u \rangle_{L^2(E)}
=
\sum_{x_i}
\left(
\sum_{x_j\sim x_i}
w_{ij}\bigl(u(x_i)-u(x_j)\bigr)
\right)
u(x_i).
\]
Since $u\in L^2(X_t)$, $u(x_i)=0$, for all $x_i\notin X_t.$ Hence, 
                            \[
\langle \delta u, \delta u \rangle_{L^2(E)}
=
\sum_{x_i\in X_t}
\left(
\sum_{x_j\sim x_i}
w_{ij}\bigl(u(x_i)-u(x_j)\bigr)
\right)
u(x_i).
\]
Now for $x_i\in X_t$,
\[
\Delta_\Gamma u(x_i)
=
\frac{1}{\mu_i}
\sum_{x_j\sim x_i}
w_{ij}\bigl(u(x_j)-u(x_i)\bigr),
\]
which implies
\[
\sum_{x_j\sim x_i}
w_{ij}\bigl(u(x_i)-u(x_j)\bigr)
=
-\mu_i\,\Delta_\Gamma u(x_i).
\]

Substituting this identity yields
\[
\langle \delta u, \delta u \rangle_{L^2(E)}
=
-\sum_{x_i}
\mu_i\,\Delta_\Gamma u(x_i)\,u(x_i)
=
-\langle \Delta_\Gamma u, u \rangle_{L^2(X_t)}.
\]

Finally, both mappings
\[
(u,v)\mapsto \langle \delta u,\delta v\rangle_{L^2(E)},
\qquad
(u,v)\mapsto \langle \Delta_\Gamma u,v\rangle_{L^2(X_t)}
\]
define symmetric bilinear forms on $L^2(X_t)$.
Therefore, the general identity follows from the polarization formula,
\[
\langle \delta u,\delta v\rangle
=
\frac14\bigl(
\langle \delta(u+v),\delta(u+v)\rangle
-
\langle \delta(u-v),\delta(u-v)\rangle
\bigr).
\]

\end{proof}

\subsection{Min-Max Characterization}

The Dirichlet eigenvalues of $\Delta_t$ satisfy the variational principle
\[
\lambda_k(\Gamma^t_{\varepsilon,\rho})
=
\min_{\substack{L \subset H^1_0(X_t) \\ \dim L = k}}
\;
\max_{u \in L \setminus \{0\}}
\frac{\|\delta u\|^2}{\|u\|^2}.
\]

\section{Average Dispersion}

In this section, we define an average dispersion functional on the truncated manifold $M_t$ for a suitable scale $r$ such that for all $x\in M_t$, the balls $B_r(x)$ do not intersect the boundary of $M$. This restriction is essential for controlling energy estimates of  compactly supported functions near the boundary.

\begin{definition}\label{def:localized-energy-Mrho}

Let $f\in L^2(M_t)$ and $0<r<2\rho<t$. For any measurable set $V\subset M_t$, define
\[
E_r^t(f,V)
:=
\int_V\int_{B_r(x)}
|f(y)-f(x)|^2\,dy\,dx.
\]

\end{definition}
We write $E_r^t(f):=E_r^t(f,M_t)$. 
\begin{lemma}\label{rem:energy-bound-Mrho} For $0<r<2\rho<t$,
there exists a constant $C_n$ such that for all
$f\in L^2(M_t)$ 
\[
E_r^t(f)
\le
C_n\nu_n r^n\,\|f\|_t^2.
\]
\end{lemma}
\begin{proof}
Using the inequality
\[
|f(y)-f(x)|^2 \le 2\bigl(|f(x)|^2+|f(y)|^2\bigr),
\]
we obtain
\begin{align*}
E_r^t(f)
&\le
2\int_{M_t}\int_{B_r(x)\cap M_t}
\bigl(|f(x)|^2+|f(y)|^2\bigr)\,dy\,dx \\
&=
4\int_{M_t}
\operatorname{vol}\bigl(B_r(x)\cap M_t\bigr)
|f(x)|^2\,dx.
\end{align*}

Since $r<2\rho<t$, for $x\in M
_t$ we have $B_r(x)\cap \partial M =\phi$. By the volume comparison theorem,
\[
\operatorname{vol}\bigl(B_r(x)\cap M_t\bigr)
\le \operatorname{vol}\bigl(B_r(x))\le  C_n\nu_n r^n
\quad\text{for all }x\in M_t.
\]
Hence,
\[
E_r^t(f)\le C_n\nu_n r^n\,\|f\|_t^2.
\]
\end{proof}
We have the following lemma due to the fact that compactly supported smooth functions are dense in $H^1_0(M_t)$.
\begin{lemma}\label{lem:Er-Mrho}

For $M\in \M$, let $0< r < 2\rho<t<\frac{i_0}{2}$. For $f\in H^1_0(M_t)$,
\[
E_r^t(f)
\le
\frac{\nu_{n}r^{n+2}}{n+2}(1+c_\MMM r)
\int_{M_t} |d f|^2\,dx,
\]where $\nu_n$ denotes the volume of the unit ball in $\mathbb{R}^n$.
\end{lemma}
\begin{proof}
 We follow the arguments in the proof of Lemma 2.3 in \cite{MaityBhattacharya2026}. Since $f$ is compactly supported in $M_t$ and $t-r>0$, the balls $B_r(x)$ do not intersect the boundary for any $x\in M_t$. Consequently, the Jacobian estimates and the local estimates along the geodesics $\exp_x(tu)$, $0\le t\le r$, used in the proof carry over to our setting.
\end{proof}
\begin{lemma}
\label{lem:poincare-Mrho}
For $0<\varepsilon<r<2\rho<t<\frac{i_0}{2}$,  let $V\subset M_{t}$ be a measurable set with $\operatorname{diam}(V)\le 2\varepsilon$ such that  $\operatorname{vol}(V)=\mu>0$. For $f\in L^2(M_t)$, let $a = \frac{1}{\mu}\int_V f(x)\,dx$ denote the average of $f$ over $V$.
Then 
\[
\int_V |f(x)-a|^2\,dx
\;\le\;
 \frac{1}{c_n(r-\varepsilon)^{n}}\,
E_r^t(f,V).
\]
\begin{proof}
    The proof follows by extending $f\in L^2(M_t)$ to the manifold by assigning it the value $0$ at all points outside $M_t.$ We then use similar arguments as in \cite{burago} as $t-r>t-2\rho>0.$ Hence, for all $x,y\in V$, $B_r(x)\cap B_r(y)$ doesn't intersect the boundary of M. The conclusion follows from the estimate $\operatorname{vol}(B_r(x))\ge C_nr^n$ for all $r<i_0$ in \cite{MR448253berger}.
\end{proof}
\end{lemma}
\section{Discretization of the manifold}
Since the boundary of the manifold is cylindrical, the Voronoi cells corresponding to discretization points near the boundary have the form $W_t \times \left[-\frac{\varepsilon}{2},\frac{\varepsilon}{2}\right]$. 
To ensure compatibility with the Dirichlet condition, we define the discretization map only for functions supported in $M_{t+\frac{\varepsilon}{2}}$. With this choice, the discretized function is assigned the value $0$ on the boundary of $X_t$.
\begin{definition}\label{def:disc-map-Mrho}
The discretization map $P \colon L^2(M_{t+\frac{\varepsilon}{2}}) \to H^1_0(X_t,\mu)$ is defined by
\[
P f(x_i)
:=
\frac{1}{\mu_i}\int_{V_i} f(x)\,dx.
\]

\end{definition}
We observe that $Pf|_{X\setminus X_{t-\varepsilon}}=0.$ For all $t>t'$, $\|Pf\|_{t'}=\|Pf\|_t.$
\begin{definition}\label{def:Pstar-Mrho}
The map $P^* \colon H^1_0(X_t) \to L^2(M_{t+\frac{\varepsilon}{2}})$ is defined by
\[
P^*u(x)
:=
\sum_{i=1}^N u(x_i)\,\mathbf 1_{V_i}(x).
\]

\end{definition}
Since the boundary is cylindrical, the support of $P^*u$ is contained in $M_{t+\frac{\varepsilon}{2}}.$ We have the following properties of $P$ and $P^*$.
\begin{lemma}\label{lem:P-properties-Mrho}
Let $P:L^2(M_{t+\frac{\varepsilon}{2}})\to L^2(X_t)$ and $P^*:H^1_0(X_t)\to L^2(M_{t+\frac{\varepsilon}{2}})$ be defined as above. Then the following hold.

\begin{enumerate}
\item For all $f\in L^2(M_{t+\frac{\varepsilon}{2}})$, $\|Pf\|_{t} \le \|f\|_{t+\frac{\varepsilon}{2}}.$

\item For all $u\in H^1_0(X_t)$, $\|P^*u\|_{t+\frac{\varepsilon}{2}} = \|u\|_{t}.$

\item For all $f\in L^2(M_{t+\frac{\varepsilon}{2}})$ and $ u\in H^1_0(X_t)$,   $\langle f,P^*u\rangle_{t+\frac{\varepsilon}{2}}
=
\langle Pf,u\rangle_{t}
 .$
\end{enumerate}
\end{lemma}

\begin{lemma}\label{lem:approx-PstarP-Mrho}
Let $\varepsilon < r < 2\rho<t.$  For every $f\in H^1_0(M_{t+\varepsilon/2})$,
\[
\|f-P^*Pf\|_{{t+\varepsilon/2}}^2
\;\le\;
\frac{4n\nu_{n}}{c_n }(1+c_\MMM r)\varepsilon^2\|d f\|_{t}^2.
\]

\end{lemma}
\begin{proof}
We have
\[
\|f-P^*Pf\|_{t+\frac{\varepsilon}{2}}^2
=
\sum_{i=1}^N
\int_{V_i}
|f(x)-Pf(x_i)|^2\,dx.
\]

For all $x_i\in X_{t+\frac{\varepsilon}{2}}$, $\operatorname{diam}(V_i)\le 2\varepsilon$ and due to our choice of points in the cylindrical boundary, $V_i\subseteq  M_{t}$. Using Lemma~\ref{lem:poincare-Mrho} we have
\[
\int_{V_i}
|f(x)-Pf(x_i)|^2\,dx
\le
\frac{1}{c_n(r-\varepsilon)^n}\,
E_r(f,V_i).
\]

Summing over $i$,
\[
\|f-P^*Pf\|_{{t+\varepsilon/2}}^2
\le
\frac{1}{c_n(r-\varepsilon)^n}
E_r^{t}(f).
\]
The rest of the arguments follow from the proof of Lemma 3.2 in \cite{MaityBhattacharya2026}.
\end{proof}

\begin{lemma}\label{lem:Pf-estimates-Mrho}
Let $t>\rho+3\varepsilon.$ For every $f\in H^1_0(M_{t+\varepsilon/2})$ the following hold:
\begin{enumerate}
\item $
 |
\|Pf\|_{t} - \|f\|_{t+\varepsilon/2}|^2
\le
\frac{4n\nu_{n}}{c_n } (1+c_\MMM r)\varepsilon^2\,\|df\|_{t+\varepsilon/2}^2. $

\item $
\|\delta(Pf)\|^2
\le
(1+c_\MMM (\rho+2\varepsilon))\left(1+\frac{2\varepsilon}{\rho}\right)^{n+2}\,\|d f\|_{t+\varepsilon/2}$.

\end{enumerate}
\end{lemma}
\begin{proof}
Since $f\in H^1_0(M_{t+\frac{\varepsilon}{2}})$ and $P^*$ is an isometry,
\[
\|Pf\|_{t} =\|Pf\|_{t+\frac{\varepsilon}{2}}= \|P^*Pf\|_{{t+\frac{\varepsilon}{2}}}.
\]
Therefore, 
\[
\big|\|Pf\|_{t+\frac{\varepsilon}{2}} - \|f\|_{t+\frac{\varepsilon}{2}}\big|\le \|P^*Pf-f\|_{t+\frac{\varepsilon}{2}}
\]
which proves $(1)$ using Lemma \ref{lem:approx-PstarP-Mrho}.

Applying Cauchy--Schwarz inequality,
\[
\big|Pf(x_j)-Pf(x_i)\big|^2
\le
\frac{1}{\mu_i\mu_j}
\int_{V_i}\int_{V_j}
|f(y)-f(x)|^2\,dy\,dx.
\]

Hence,
\[
\|\delta(Pf)\|^2
\le
\frac{n+2}{\nu_n\rho^{n+2}}
\sum_{x_i\sim x_j}
\int_{V_i}\int_{V_j}
|f(y)-f(x)|^2\,dy\,dx.
\]

For $x\in V_i$,  let $U(x):=\bigcup_{x_j\sim x_i} V_j
\subseteq B_{\rho+2\varepsilon}(x).$
For all $Pf\in H^1_0(X_{t})$ the voronoi cells $V_i$ corresponding to non-zero entries lie in $M_{t+\frac{\varepsilon}{2}}.$ Since $t>2\rho>\rho+2\varepsilon$, $B_{\rho+2\varepsilon}(x)\cap\partial M=\phi.$ \\Hence,
\begin{align*}
   \|\delta(Pf)\|^2
&\le
\frac{n+2}{\nu_n\rho^{n+2}}
\int_{M_{t}}
\int_{B_{\rho+2\varepsilon}(x)}
|f(y)-f(x)|^2\,dy\,dx \\
&=
\frac{n+2}{\nu_n\rho^{n+2}}\,E^{t}_{\rho+2\varepsilon}(f). 
\end{align*}

Applying Lemma~\ref{lem:Er-Mrho}, since $t>\rho+2\varepsilon$,
\[
E^t_{\rho+2\varepsilon}(f)
\le
\frac{\nu_n}{n+2}
(\rho+2\varepsilon)^{n+2}
(1+c_\M (\rho+2\varepsilon))\,\|d f\|_{{t}}^2.
\]

Therefore,
\[
\|\delta(Pf)\|^2
\le
\left(\frac{\rho+2\varepsilon}{\rho}\right)^{n+2}
(1+c_\M (\rho+2\varepsilon))\,\|d f\|_t^2.
\]

\end{proof}
\begin{theorem}\label{prop:upper-bound-Mrho}
Let $0<\varepsilon<2\rho< t<s_0<\frac{i_0}{2}$, $M\in \M$ and $\Gamma^t_{\varepsilon,\rho}$ be an $(\varepsilon,\rho)$-proximity graph on $M$. 
Let $\lambda_k(M_{t+\varepsilon})$ and $\lk(\Gamma^t_{\varepsilon,\rho})$ denote the $k$--th Dirichlet eigenvalues of $-\Delta$ on $M_{t+\varepsilon}$ and $\Delta_t$ respectively.  There exists positive constants $C_n$ and $C_{\M,k}$ such that for any $\rho,\frac{\varepsilon}{\rho}<\frac{1}{C_n}$,
\[
\lambda_k(\Gamma^t_{\varepsilon,\rho})
\;\le\;
\left(1+C_{\M,k}\left(\frac{\varepsilon}{\rho}+\rho\right)\right)\lk(M_{t+\varepsilon}).
\]

\end{theorem}

\begin{proof}
For a subspace $L\subset H^1_0(X_{t})$ with $\dim L=k$, if
\[
\sup_{0\neq u\in L}
\frac{\|\delta u\|^2}{\|u\|_{t}^2}
\le
\left(1+C_{\M,k}\left(\frac{\varepsilon}{\rho}+\rho\right)\right)\lk(M_{t+\varepsilon})
\]
then by the min-max principle for $\Delta_t$, the conclusion follows.\\
Let $W\subset H_0^1(M_{t+\varepsilon/2})$ be the $k$-dimensional subspace spanned by the first $k$ orthonormal Dirichlet eigenfunctions of $-\Delta$ on $M_{t+\varepsilon/2}$.
Then for all $f\in W$, since $M_{t+\varepsilon}\subseteq M_{t+\varepsilon/2}$, applying domain monotonicity on eigenvalues
\[
\|d f\|_{t+\varepsilon/2}^2
\le
\lambda_k(M_{t+\varepsilon})\,\|f\|_{t+\varepsilon/2}^2.
\]
Let $L := P(W)\subset H^1_0(X_t).$   
Using Lemma~\ref{lem:approx-PstarP-Mrho},  $P$ is injective on $W$ for $\rho<\frac{1}{C_{\M,k}}$, where the constant is obtained from \cite{cheng} which asserts the existence of an upper bound of $\lk(M_{t+\varepsilon})$  depending only on $\M$ and $k$. 

Let $u\in L\setminus\{0\}$ and choose $f\in W$ such that $u=Pf$. 
We use Lemma \ref{lem:Pf-estimates-Mrho} to estimate the Rayleigh quotient of $u$, by following the arguments of the proof of Theorem 3.5 in \cite{MaityBhattacharya2026} which carry over to the Dirichlet setting.

\end{proof}

\section{Smoothing operator}

In this section, we present a smoothing operator that maps $L^2$ functions on $M_{t+\varepsilon/2}$ to $C^{0,1}$ functions on $M_{t+\varepsilon/2-r}$. The operator spreads the function locally over balls of radius $r$ in $M$ using a kernel, while avoiding any intersection with the boundary.
\begin{definition}
  Let  $\psi\colon [0,\infty)\to\mathbb{R}$ by
\[
\psi(t)
:=
\begin{cases}
\displaystyle
\frac{n+2}{2\nu_n}\,(1-t^2), & 0\le t\le 1,\\[0.3em]
0, & t>1,
\end{cases}
\]
where $\nu_n$ denotes the volume of the unit ball in $\mathbb{R}^n$.  
\end{definition}

The normalization constant is chosen so that
\[
\int_{\mathbb{R}^n}\psi(|x|)\,dx = 1.
\]
\begin{definition}
  Fix $r$ such that $0<r<2\rho<t$.  For $x\in M$, such that $d(x,\partial M)>r$,  we define the kernel as
\[
k_r(x,y)
:=
r^{-n}\,\psi\!\left(\frac{d(x,y)}{r}\right).
\]
 
\end{definition}

Due to the choice of $r$, the support of $k_r$ doesn't intersect the boundary of $M$. Then, \[\mathrm{grad}k_r(.,y)(x)=\frac{n+2}{\nu_nr^{n+2}}\exp^{-1}(y), \qquad \forall x\in M_t \text{ and } y\in B_r(x).\]

\begin{lemma}
\label{eq:kernel_symmetric}
The kernel map has the following properties.
    \begin{enumerate}
        \item For $x\in M_r$, $    \big| k_r(x,y)\big|\le\frac{n+2}{\nu_nr^n}.$
\item For $d(x,\partial M)>2r$, $k_r(x,y)=k_r(y,x).$

    \end{enumerate}
\end{lemma}

\begin{definition}\label{def:smoothing-Mrho}
For $0<r<2\rho<t$ the integral operator $\Lambda_r^0 \colon L^2(M_{t+\varepsilon/2})\to C^{0,1}(M_{t+\varepsilon/2-r})$ is defined
by
\[
\Lambda^0_r f(x)
:=
\int_{M_{t+\varepsilon/2}} k_r(x,y)\,f(y)\,dy.
\]
\end{definition}
\begin{definition}
   For $x\in M$ such that $d(x,\partial M)>r$, define $\theta_r(x)=\int_{B_r(x)}k_r(x,y)dy$ i.e.,
\[
\theta_r(x)=r^{-n}\int_{B_r(x)} \psi\!\left(\frac{d(x,y)}{r}\right)\,dy.
\]
 
\end{definition}

\begin{definition}\label{def:normalized-smoothing}
Let $0<r<2\rho<t$. The normalized smoothing operator $
\Lambda_r \colon L^2(M_{t+\varepsilon/2})\to C^{0,1}(M_{t+\varepsilon/2-r})
$ is defined as 
\[
\Lambda_r f(x)
:=
\frac{1}{\theta_r(x)}\,\Lambda^0_r f(x).
\]
\end{definition}
\begin{remark}
The restriction that $r<2\rho<t$ ensures that the smoothing kernel is supported away from the boundary so that all Jacobian and gradient estimates hold exactly as in the case witohut boundary. We state Lemma 4.2 and Lemma 4.3 from \cite{MaityBhattacharya2026}.
\end{remark}
\begin{lemma}\label{lem:theta-Mrho}
Let $0<r<2\rho<t<\frac{i_0}{2}$. For all $x\in M_t$,
\begin{enumerate}
    \item $1-c_\M r \le \theta_r(x) \le 1+ c_\M r,$ and
    \item $|\nabla \theta_r(x)| \le c_\M .$
\end{enumerate}
\end{lemma}

\begin{lemma}\label{lem:L2-stability-Mrho}
Let $0<3r<t<\frac{i_0}{2}$. For every $f\in L^2(M_{t+\varepsilon/2})$ and $r<\frac{1}{c_\M}$,
\[
\|\Lambda_r f\|^2_{t+\frac{\varepsilon}{2}-r}
\le
\left( \frac{1+c_\MMM r}{1-c_\MMM r}\right)\,\|f\|^2_{t+\varepsilon/2}.
\]
\end{lemma}
\begin{proof}
Let $x\in M_{t+\varepsilon/2-r}.$ Since $t+\varepsilon/2-r>r$, $B_r(x)\cap \partial M=\phi$ and $B_r(x)\cap M_{t+\varepsilon/2-r}\subseteq B_r(x).$ Hence,
\[\int_{M_{t+\varepsilon/2}}k_r(x,y)dy\le \int_{B_r(x)}k_r(x,y)dy=\theta_r(x).\]
By Cauchy-Schwarz inequality,
\begin{align} 
|\Lambda^0_r f(x)|^2 &\le \left(\int_{M_{t+\varepsilon/2}}k_r(x,y)dy\right)\left(\int_{M_{t+\varepsilon/2}} |f(y)|^2 k_r(x,y)\,dy\right)\nonumber\\ &\le\theta_r(x)\int_{M_{t+\varepsilon/2}} |f(y)|^2 k_r(x,y)\,dy.\label{eq:upper_bound_on_integral_operator}
\end{align}

Integrating over $x\in M_{t+\frac{\varepsilon}{2}-r}$  and substituting (\ref{eq:upper_bound_on_integral_operator}) gives
\begin{align}
  \|\Lambda_r f\|_{t+\frac{\varepsilon}{2}-r}^2
&\le \int_{M_{t+\frac{\varepsilon}{2}-r}}\theta_r(x)^{-1}
\int_{M_{t+\frac{\varepsilon}{2}}}|f(y)|^2
k_r(x,y)\,dy\,dx \nonumber\\ 
&\le (1-c_\M r)^{-1}\int_{M_{t+\frac{\varepsilon}{2}}}|f(y)|^2
\int_{M_{t+\frac{\varepsilon}{2}-r}}
k_r(x,y)\,dx\,dy.\nonumber  
\end{align}

Since $t>3r$, we have $t+\frac{\varepsilon}{2}-r>2r$. Using Lemma \ref{eq:kernel_symmetric}, 
\[
\int_{M_{t+\frac{\varepsilon}{2}-r}}
k_r(x,y)\,dx\le \int_{B_r(y)}
k_r(x,y)\,dx=\theta(y)\le (1+c_\M r).
\]
The result follows.
\end{proof}
\begin{lemma}\label{lem:smoothing-error-Mrho}
Let $0<3r<t<\frac{i_0}{2}$. For every $f\in L^2(M_{t+\varepsilon/2})$ and $r<\frac{1}{c_\M}$,
\[
\|\Lambda_r f - f\|_{t+\varepsilon/2-r}^2
\le
\frac{n+2}{\nu_nr^{n}\lbtheta }
E_r^{t-\varepsilon/2-r}(f).
\]
\end{lemma}
\begin{proof}
Fix $x\in M_{t+\varepsilon/2}$.  Since $t>3r$, $B_r(x)$ doesn't intersect the boundary of the manifold $M$. Hence, the support of $\theta_r(x)$ and the kernel $k_r$ at $x$ does not intersect the boundary. We have
\[
\Lambda_r f(x)-f(x)
=
\theta_r(x)^{-1}
\int_{B_r(x)}(f(y)-f(x))k_r(x,y)\,dy.
\]
By Cauchy--Schwarz inequality,
\begin{align}
 |\Lambda_r f(x)-f(x)|^2
&\le \theta_r(x)^{-2}\left(\int_{B_r(x)}k_r(x,y)dy\right)\left(\int_{B_r(x)} |f(y)-f(x)|^2 k_r(x,y)\,dy\right) \nonumber\\
&=\theta_r(x)^{-1}
\int_{B_r(x)}|f(y)-f(x)|^2 k_r(x,y)\,dy \nonumber\\
&\le \frac{n+2}{\nu_nr^{n}\lbtheta }\int_{B_r(x)}|f(y)-f(x)|^2 \label{eq:theta_bound_and_kernel_bound}
\end{align}
where (\ref{eq:theta_bound_and_kernel_bound}) follows from Lemma \ref{lem:theta-Mrho} and Lemma \ref{eq:kernel_symmetric}.
\end{proof}
We state the bound from Lemma 4.6 of \cite{MaityBhattacharya2026} in our setting for compactly supported functions. Since $B_r(x)\cap\partial M=\varnothing$ for all $x\in M_{t+\varepsilon/2}$, the estimates in \cite{MaityBhattacharya2026} apply in the present setting.
\begin{lemma} \label{lem:upper bound on dlambda}
Let $0<r<2\rho<i_0 \text{ and } t>3r$. For every $f\in L^2(M_{t+\frac{\varepsilon}{2}}),$
\[\|d(\Lambda_rf)\|^2_{t+\varepsilon/2-r}\le \frac{n+2}{\nu_{n}r^{n+2}} \frac{(1+ 2c_\MMM r)^2}{(1-c_\MMM r)^3
     }E_r^{t+\varepsilon/2-r}(f).\]
\end{lemma}

\section{Interpolation map}

To compare functions on the discretization $X_t$ with functions on the manifold, we introduce an interpolation map that lifts discrete functions to Lipschitz functions on $M$. The map is obtained by first extending a function on $X_t$ to the manifold using the operator $P^*$ and then applying the smoothing operator $\Lambda$. This construction will allow us to relate the discrete Dirichlet energy to the continuous Dirichlet energy on the manifold.
\begin{definition}
The interpolation map $I:H^1_0(X_t)\rightarrow C^{0,1}(M_{t+2\varepsilon-\rho})$
is defined by
\[
Iu=\Lambda_{\rho-2\varepsilon}P^*u.
\]
\end{definition}
For $u\in H^1_0(X_t)$, the support of $Iu$ is contained in 
$M_{t+\frac{5\varepsilon}{2}-\rho}\subseteq M_{t+2\varepsilon-\rho}$. Since $\|Iu\|_{t+\frac{5\varepsilon}{2}-\rho}=\|Iu\|_{t+2\varepsilon-\rho}$, 
we use the latter notation for simplicity.

From Lemma  and the fact that $P^*$ preserves the norm, we have \begin{equation*}
 \norm{Iu}_{t+2\varepsilon-\rho}\leq \left(\frac{1+
 c_\MMM (\rho-2\varepsilon)}{1-c_\MMM(\rho-2\varepsilon)}\right)\norm{u}_{t}.  
\end{equation*}
\begin{lemma}
\label{lem:interpolation}For $0<2\varepsilon<\rho<3\rho<t$ , $c_\MMM \rho<\frac{1}{4}$ and $\frac{2\epsilon}{\rho}<\frac{1}{n}$, let $u\in H^1_0(X_t)$. Then
    \begin{enumerate}
        \item $|\|Iu\|_{t+2\varepsilon-\rho}-\|u\|_{t}|\le \frac{3\rho^2}{1-c_{\MMM}\rho} \|\delta u\|$;
        \item $\|d(Iu)\|_{t+2\varepsilon-\rho}\le \left(1-\frac{2\epsilon}{\rho}\right)^{-\frac{n}{2}-1} \frac{1+ 2c_\MMM \rho}{(1-c_\MMM \rho)^\frac{3}{2}}\|\delta u\|.$

    \end{enumerate}
\end{lemma}
\begin{proof}
   For  $u\in H^1_0(X_t)$, using the fact that $P^*$ preserves the norm and $2\varepsilon<\rho\implies M_{t+\frac{\varepsilon}{2}}\subseteq M_{t+2\varepsilon-\rho}$,
   \[|\|Iu\|_{t+2\varepsilon-\rho}-\|u\|_{t}|=\big|\norm{Iu}_{t+2\varepsilon-\rho}-\norm{P^*u}_{t+2\varepsilon-\rho}\big|\le\|Iu-P^*u\|_{t+2\varepsilon-\rho}.\]
 From Lemma \ref{lem:smoothing-error-Mrho}   \begin{equation}
 \label{eq:upper bound on Iu-Pu}
   \|Iu-P^*u\|^2_{t+2\varepsilon-\rho}\le\frac{n+2}{\nu_{n}(1-c_\MMM(\rho-2\epsilon))(\rho-2\epsilon)^{n}} E_{\rho-2\varepsilon}^{t+2\varepsilon-\rho}(P^*u). 
 \end{equation}
 The Dirichlet energy of $u$ when calculated on $M_{t+2\varepsilon-\rho}$ is
 \[  \|\delta u\|^2=\frac{n+2}{\nu_n \rho^{n+2}}\sum_{x_i\in X}\sum_{j:x_i\sim x_j}\mu_i\mu_j|u(x_j)-u(x_i)|^2.  \]
 Since, $t>2\rho$, the edges with one endpoint in $\partial M$ cannot have another endpoint in $X_t$. Hence, these endpoints are not in the support of $u$ and they don't contribute to the Dirichlet energy. For $x\in V_i$, let $U(x):=\bigsqcup_{j:x_j\sim x_i}V_j$. 
 Since $t>2\rho>\varepsilon$ and $u\in H^1_0(M_t), $ for any $x\in M\setminus M_\varepsilon$, $P^*u|_{U(x)}\equiv 0.$ Hence, 
 \begin{equation}
  \norm{\delta u}^2=\frac{n+2}{\nu_n \rho^{n+2}}\int_{M_{\varepsilon}}\int_{U(x)}|P^*u(y)-P^*u(x)|^2\,dy\,dx. \label{eq:del u lower bound}   
 \end{equation}

By triangle inequality, for all $y\in B_{\rho-2\varepsilon}(x)$, $d(x_i,x_j)\le d(x,x_i)+d(x,y)+d(y,x_j)\le \rho. $ Hence, $B_{\rho-2\varepsilon}(x)\subseteq U(x).$ From (\ref{eq:del u lower bound}),
\begin{align}
        \|\delta u\|^2&\ge \frac{n+2}{\nu_n \rho^{n+2}}\int_{M_{\varepsilon}}\int_{B_{\rho-2\varepsilon}(x)}|P^*u(y)-P^*u(x)|^2\,dy\,dx\nonumber \\
    &= \frac{n+2}{\nu_n \rho^{n+2}}E^{\varepsilon}_{\rho-2\varepsilon}(P^*u). \label{eq:del_u_bound_dispersion}
\end{align}

This implies, from (\ref{eq:upper bound on Iu-Pu}),
\begin{align}
    \|Iu-P^*u\|^2_{t+2\varepsilon-\rho}&\le\frac{n+2}{\nu_{n}(1-c_\MMM(\rho-2\epsilon))(\rho-2\epsilon)^{n}}E_{\rho-2\varepsilon}^{t+2\varepsilon-\rho}(P^*u)\nonumber\\
    &\le \frac{n+2}{\nu_{n}(1-c_\MMM(\rho-2\epsilon))(\rho-2\epsilon)^{n}}E_{\rho-2\varepsilon}^{\varepsilon}(P^*u)\nonumber\\ &\le \frac{3\rho^2}{1-c_{\MMM}\rho}\|\delta u\|^2 \nonumber
\end{align}

which proves $(1)$.\\
From Lemma \ref{lem:upper bound on dlambda} and (\ref{eq:del_u_bound_dispersion}),
\[\|d(Iu)\|_{t+2\varepsilon-\rho}^2\le \frac{n+2}{\nu_{n}r^{n+2}} \frac{(1+ 2c_\MMM r)^2}{(1-c_\MMM r)^3
     }E_{\rho-2\varepsilon}^{t+2\varepsilon-\rho}(P^*u)\le \left(\frac{\rho}{\rho-2\varepsilon}\right)^{n+2}\frac{(1+ 2c_\MMM r)^2}{(1-c_\MMM r)^3
     }\|\delta u\|^2.\]

\end{proof}
\begin{theorem}
\label{thm:lower_bound}
Let $0<\varepsilon<2\rho< t<s_0<\frac{i_0}{2}$, $M\in \M$ and $\Gamma^t_{\varepsilon,\rho}$ be an $(\varepsilon,\rho)$-proximity graph on $M$. 
Let $\lambda_k(M_{t+2\varepsilon-\rho})$ and $\lk(\Gamma^t_{\varepsilon,\rho})$ denote the $k$--th Dirichlet eigenvalues of $-\Delta$ on $M_{t+2\varepsilon-\rho}$ and $\Delta_t$ respectively.  There exists positive constants $C_n$ and $C_{\M,k}$ such that for any $\rho,\frac{\varepsilon}{\rho}<\frac{1}{C_n}$,
\[
\lambda_k(\Gamma^t_{\varepsilon,\rho})
\;\ge\;
\left(1+C_{\M,k}\left(\frac{\varepsilon}{\rho}+\rho\right)\right)\lk(M_{t+2\varepsilon-\rho}).
\]

\end{theorem}
\begin{proof}
    The proof follows similar arguments as in Theorem \ref{prop:upper-bound-Mrho}. 
    For a subspace $L\subset H^1_0(M_{t+2\varepsilon-\rho})$ with $\dim L=k$, if
\[
\sup_{0\neq f\in L}
\frac{\|df\|^2}{\|f\|_{t+2\varepsilon-\rho}^2}
\le
\left(1+C_{\M,k}\left(\frac{\varepsilon}{\rho}+\rho\right)\right)^{-1}\lk(\Gamma^t_{\varepsilon,\rho})
\]
then by the min-max principle for $-\Delta$, the conclusion follows.\\
Let $W\subset H_0^1(X_{t})$ be the $k$-dimensional subspace spanned by the first $k$ orthonormal Dirichlet eigenfunctions of $-\Delta_t$.
Then for all $u\in W$, $\|\delta u\|^2
\le
\lambda_k(\Gamma^t_{\varepsilon,\rho})\,\|u\|_{t}^2.$
Let $L := I(W)\subset H^1_0(M_{t+2\varepsilon-\rho}).$   
Using Lemma~\ref{lem:interpolation},  $I$ is injective on $W$ for $\rho<\frac{1}{C_{\M,k}}$.

Let $f\in L\setminus\{0\}$ and choose $u\in W$ such that $f=Iu$. 
We then use Lemma \ref{lem:interpolation} to estimate the Rayleigh quotient of $f$ and the result follows.

\end{proof}
\textbf{Proof of Theorem \ref{thm:main}} Combining the estimates obtained from Theorem \ref{prop:upper-bound-Mrho} and Theorem \ref{thm:lower_bound}, we obtain the required comparison.\\

\textbf{Proof of Corollary \ref{corollary}} The proof follows from Theorem \ref{thm:main}, domain monotonicity of Dirichlet eigenvalues, 
 and the fact that for a monotone sequence of truncated submanifolds of $M$, by multiplying the eigenfunctions with cut-off functions and using min-max principle, we have $\lim_{\varepsilon,\rho\to 0}\lk(M_{t+2\varepsilon-\rho})=\lk(M_t)$ and $\lk(M_{t+\varepsilon})=\lk(M_t)$.  \\

 \textbf{Proof of Proposition \ref{proposition}} The proof follows from the domain monotonicity and the arguments used in the proof of Corollary \ref{corollary}.

\section*{Acknowledgement}
I am grateful to my advisor, Dr. Soma Maity, for introducing me to this problem and for the insightful discussions and helpful guidance over the years. I acknowledge the financial support provided by the ``Council of Scientific and Industrial Research (CSIR), India." The results in this paper were obtained during the author’s Ph.D. at the Indian Institute of Science Education and Research (IISER), Mohali, and were a part of the author’s dissertation.

\bibliographystyle{plain}
\bibliography{sample}
\end{document}